\documentclass[12pt]{amsart}
\usepackage{amsthm,amsfonts,amsmath,amssymb,amsaddr,geometry,eufrak,bigints,eucal}
\usepackage{graphicx}
 \usepackage{enumitem}
 \usepackage{hyperref}
 \usepackage{etoolbox}
 \usepackage{doi}
 \usepackage{tikz}
 \usetikzlibrary{calc, arrows.meta, positioning}

 \hypersetup{colorlinks=true, linkcolor=blue, citecolor=red}
 \makeatletter
\patchcmd{\section}{\scshape}{\relax}{}{}
\patchcmd{\section}{\uppercase}{\relax}{}{}
\makeatother
\input xypic
\usepackage{url}
\newtheorem{theorem}{Theorem}[section]
\newtheorem{proposition}[theorem]{Proposition}

\newtheorem{corollary}[theorem]{Corollary}
\newtheorem{definition}[theorem]{Definition}
\newtheorem{remark}[theorem]{Remark}
\newtheorem{example}[theorem]{Example}

\begin{document}
\title[]{Four Point Analogue of \'Ciri\'c--Reich--Rus Type Mappings with Non-Unique Fixed Points}

\author{Irfan Ahmed$^1$, Shallu Sharma$^2$, Sahil Billawria$^{3,*}$}
\maketitle

\begin{center}
\footnotesize $^1$Department of Mathematics, University of Jammu, Jammu, India.\\
Email: irfanahmed0038@gmail.com\\
 $^2$Department of Mathematics, University of Jammu, Jammu, India.\\
Email: shallujamwal09@gmail.com\\
 $^3$Department of Mathematics, University of Jammu, Jammu, India.\\
Email: sahilbillawria2@gmail.com
\end{center}
\begin{abstract}
In this paper, we introduce and study a new class of mappings termed generalized quadrilateral \'Ciri\'c-Reich-Rus type mappings, which serve as a four-point analogue of \'Ciri\'c-Reich-Rus type mappings. We establish the existence and non-uniqueness of fixed points, with at most three fixed points possible for these mappings in complete metric spaces under the condition that the mapping has no periodic points of prime period $2$ and $3$. We determine that while these mappings are often discontinuous within their domain of definition, they must be continuous at their fixed points. By applying additional constraints, notably asymptotic regularity and continuity, we broaden the applicability of fixed-point theorems to include a broader class of mappings. At the end, we attain fixed point theorems for generalized quadrilateral \'Ciri\'c-Reich-Rus type mappings in metric spaces that are not necessarily complete. Several illustrative examples are provided to validate the theoretical results and demonstrate the independence of our new class from existing notions.
\end{abstract}
\section{Introduction and Preliminaries}

Fixed point theory is an important area of mathematics. It is widely used in differential equations, integral equations, optimisation and many branches of applied sciences. The basic problem is to decide whether a given map $T$ on a space $X$ has a point $x\in X$ with $Tx=x$. When $T$ is a contraction on a complete metric space, the classical Banach contraction principle \cite{Banach} guarantees the existence and uniqueness of such a point. This result has been extended by many researchers since its inception, by relaxing the contraction criterion or by exploring alternative contexts.

Among the notable generalizations are the results of Kannan \cite{Kannan} and Chatterjea \cite{Chatterjea}, who introduced contractive conditions that do not require the mapping to be continuous. These conditions are independent of Banach's contraction and, in fact, characterize completeness of the underlying metric space \cite{subra}. Later in 1971,  \'Ciri\'c \cite{Ciric}, Reich \cite{Reich1} and Rus \cite{Rus} combined the Banach, Kannan and Chatterjea type inequalities into a single unified condition known as the Ćirić-Reich-Rus (CRR) contraction that provides the fixed point for a self discontinuous mappings.
\begin{definition}[Reich \cite{Reich1}]
Let $T: X \to X$ be a mapping on a complete metric space $(X,d)$ with
\begin{equation}
d(Tx, Ty) \le a\,d(x,y) + b\,d(x,Tx) + c\,d(y,Ty), \tag{1.1}\label{1.1}
\end{equation}
where $a,b,c \ge 0$, $a+b+c<1$ and $x,y \in X$. Then $T$ has a unique fixed point.
\end{definition}

There are various classical results in the literature of fixed point theory. Key generalisations of fixed point results include weakening contractivity \cite{2012, 2003, 1977, 2021, 2020}, relaxing topological requirements or metric spaces \cite{1959, 1990, 2009}, and extending theorems to cope with multivalued mappings \cite{1969, 2011}. These modifications increase the applicability of fixed point theorems to a variety of mathematical contexts. Recently, a new geometric direction appeared. Petrov \cite{Petrov1} introduced mappings which contract the perimeter of triangles and proved a fixed point theorem under the assumption that the mapping has no periodic points of prime period 2. Such a mapping need not have a unique fixed point, in fact it can have at most two fixed points. Let us recall that.

\begin{definition}[Petrov \cite{Petrov1}]
Let $(X,d)$ be a metric space with $\left|X\right|\ge 3$. Then the mapping $T:X\to X$ is defined as contracting perimeters of triangles if there is an $\alpha \in [0,1)$ such that
\begin{equation}
d(Tp,Tq) + d(Tq,Tr) + d(Tr,Tp) \le \alpha \bigl(d(p,q) + d(q,r) + d(r,p)\bigr) \tag{1.2}
\end{equation}
for three pairwise distinct points $p,q,r\in X$.
\end{definition}

This idea was further extended by Petrov and Bisht \cite{PetrovBisht1} and by Bisht and Petrov \cite{BishtPetrov, BishtPetrovCRR} where three point analogues of Kannan, Chatterjea and Ćiric-Reich-Rus type contractions were introduced respectively. In each case the absence of periodic points of prime period 2 is essential and the number of fixed points is atmost 2. 
\begin{definition}[Bisht--Petrov \cite{BishtPetrovCRR}]
Let $(X,d)$ be a metric space with $|X| \ge 3$. We shall say that $T:X\to X$ is a generalized \'CRR type mapping on $X$ if there exist $\alpha,\lambda \ge 0$, $2\alpha + \frac{3\lambda}{2} < 1$ such that the inequality
\begin{equation}
\begin{aligned}
&d(Tp,Tq) + d(Tq,Tr) + d(Tp,Tr) \\
&\le \alpha\bigl(d(p,q) + d(q,r) + d(r,p)\bigr) + \lambda\bigl(d(p,Tp) + d(q,Tq) + d(r,Tr)\bigr)
\end{aligned} \tag{1.3}\label{1.5}
\end{equation}
holds for all three pairwise distinct points $p,q,r\in X$.
\end{definition}
In 2025, Petrov \cite{petrov2025} extended this direction of research to mappings which contract the total pairwise distances between $k$ points in metric spaces. Several additional works that utilise multipoint Banach contractions are available in \cite{bisht2025, bishtpetrov2025,  petrovsamet} including those that were developed within the semimetric space framework based on the triangle functions introduced by Bessenyei and Páles \cite{bessen}.

More recently, Banerjee et al. \cite{Banerjeeetal} introduced a four-point analogue of Banach, Kannan, and Chatterjea type contractions.
\begin{definition}[Banach-type]
Let $(X,d)$ be a metric space with $\left|X\right|\ge 4$. Then a mapping $T \colon X \to X$ is said to be a perimetric contraction on quadrilaterals in $X$ if there is an $\alpha \in [0, 1)$ such that
\begin{equation}
d(Tp, Tq) + d(Tq, Tr) + d(Tr, Ts) + d(Ts, Tp) \leq \alpha \big( d(p, q) + d(q, r) + d(r, s) + d(s, p) \big)\tag{1.4}
\end{equation}
for all distinct points $p, q, r, s \in X$.
\end{definition}
\begin{definition}[Kannan-type]
Let $(X,d)$ be a metric space with $\left|X\right|\ge 4$. Then a mapping $T \colon X \to X$ is said to be a Kannan-type perimetric contraction on quadrilaterals in $X$ if there is an $\alpha \in \left[0, \frac{1}{2}\right)$ such that
\begin{equation}
 d(Tp, Tq) + d(Tq, Tr) + d(Tr, Ts) + d(Ts, Tp)\leq \alpha \big( d(p, Tp) + d(q, Tq) + d(r, Tr) + d(s, Ts) \big)\tag{1.5}
\end{equation}
for all distinct points $p, q, r, s \in X$.
\end{definition}

\begin{definition}[Chatterjea-type]
Let $(X,d)$ be a metric space with $\left|X\right|\ge 4$. Then a mapping $T \colon X \to X$ is said to be a Chatterjea-type perimetric contraction on quadrilaterals in $X$ if there is an $\alpha \in \left[0, \frac{1}{7}\right)$ such that
\begin{align*}
&d(Tp, Tq) + d(Tq, Tr) + d(Tr, Ts) + d(Ts, Tp)\\
&\quad\leq \alpha \big( d(p, Tq) + d(q, Tp) + d(r, Ts) + d(s, Tr)\\
&\quad\quad+d(p,Ts)+d(s,Tp)+d(q, Tr)+d(r,Tq) \big)\tag{1.6}
\end{align*}
for all distinct points $p, q, r, s \in X$.
\end{definition}

Motivated by these developments, in the next definition, we introduce the four-point analogue of the CRR-type mappings.
\begin{definition}
Let $(X,d)$ be a metric space with at least four points. A mapping $T:X\to X $ is called a generalized quadrilateral  Ćiric-Reich-Rus type (GQCRR) mapping if there exist $\alpha,\lambda \ge 0$ with 
\[
3\alpha+2\lambda<1
\]
such that for all four pairwise distinct points $p,q,r,s\in X$,
\begin{align*}
&d(Tp,Tq)+d(Tq,Tr)+d(Tr,Ts)+d(Ts,Tp)\\
&\quad \le \alpha\big(d(p,q)+d(q,r)+d(r,s)+d(s,p)\big)\\
&\quad\quad+\lambda\big(d(p,Tp)+d(q,Tq)+d(r,Tr)+d(s,Ts)\big).\tag{1.7}\label{0.1}
\end{align*}
\end{definition}
Our main aim  is to prove fixed point theorems for this new class of mappings. We show that in a complete metric space, such a mapping has a fixed point provided it has no periodic points of prime period $2$ and $3$. Moreover, the number of fixed points is at most three. We further assess the regularity properties of these mappings, establishing that even if they are discontinuous across their domain, they necessarily exhibit continuity at fixed point.

We further extended our findings in two ways. In order to reduce the parameter constraints while ensuring the existence of fixed points, we first examine asymptotically regular GQCRR mappings. Secondly, we derive fixed point theorems in incomplete metric spaces. In one scenario, we assume continuity at the limit point and in another, we assume that the mapping is continuous on the entire space and that the contraction condition holds on a dense subset.

Throughout the paper, we have denoted $(X,d)$ as metric space, $|X|$ as the cardinality of the set $X$, $\mathbb{N}$ as the set of natural numbers, and $\operatorname{Fix}(T)$ as the collection of all fixed points of $T$. Let $T:X\to X$ be a mapping. A point $p \in X$ is said to be a fixed point of $T$ if $T(p)=p$; $p$ is said to be a periodic point of period $n$ if $T^n p = p$. The least positive integer $n$ for which $T^n p = p$ is called the prime period of $p$ \cite{periodicpoint}.

\section{Some Properties of Generalized Quadrilateral CRR-Type Mappings}

\begin{proposition}\label{prop1}
Let $(X, d)$ be a metric space with $|X| \ge 4$. If $T:X \rightarrow X$ is a CRR type mapping satisfying $3a + 2(b+c) < 1$, then $T$ is a generalized quadrilateral CRR type mapping.
\end{proposition}
\begin{proof}
Let $p,q,r,s\in X$ be any four pairwise distinct points. Apply the inequality (\ref{1.1}) to  the pairs $(q,r), (r,s)$ and $(s,p)$:
\begin{align*}
d(Tq,Tr) &\le a\,d(q,r) + b\,d(q,Tq) + c\,d(r,Tr),  \tag{2.1}\label{2.1}\\ 
d(Tr,Ts) &\le a\,d(r,s) + b\,d(r,Tr) + c\,d(s,Ts), \tag{2.2}\label{2.2}\\
d(Ts,Tp) &\le a\,d(s,p) + b\,d(s,Ts) + c\,d(p,Tp). \tag{2.3}\label{2.3}
\end{align*}
By summing the left-hand and right-hand sides of inequalities (\ref{1.1}), (\ref{2.1}), (\ref{2.2}) and (\ref{2.3}), we obtain
\begin{align*}
&d(Tp,Tq)+d(Tq,Tr)+d(Tr,Ts)+d(Ts,Tp) \\
&\quad \le a\big(d(p,q)+d(q,r)+d(r,s)+d(s,p)\big) \\
&\quad \quad + (b+c)\big(d(p,Tp)+d(q,Tq)+d(r,Tr)+d(s,Ts)\big).\tag{2.4}\label{2.4}
\end{align*}
Set \(\alpha := a\) and \(\lambda := b+c\). Then (\ref{2.4}) is precisely the inequality (\ref{0.1}).
\end{proof}

\begin{example}
Let $X = [0,1]$ be equipped with the Euclidean metric $d$. Let $T:X \rightarrow X$ be defined as:
\begin{equation*}
Tx = 
\begin{cases} 
\frac{p}{8}, & \text{if } p \in [0,1), \\
\frac{1}{16}, & \text{if } p = 1.
\end{cases}
\end{equation*}

Clearly, $\lim_{p \to 1^-} Tp = \frac{1}{8} \neq T(1)$, making the mapping discontinuous at $p = 1$. We will first show that $T$ is a  ĆRR-type mapping with the coefficients $a = \frac{1}{8}$ and $b = c = \frac{1}{16}$. We must verify that (\ref{1.1}) holds for all $p, q \in [0,1]$.

\textbf{Case 1:} Let $p, q \in [0,1)$, applying (\ref{1.1}) yields\\
\begin{align*}
\left|\frac{p}{8} - \frac{q}{8}\right| &\le \frac{1}{8}|p - q| + \frac{1}{16}\left(p - \frac{p}{8}\right) + \frac{1}{16}\left(q - \frac{q}{8}\right) \\
\left|\frac{p}{8} - \frac{q}{8}\right| &\le \frac{1}{8}|p - q| + \frac{7p}{128} + \frac{7q}{128}.
\end{align*}
Since $p, q \ge 0$, the additional terms on the right side are non-negative, and the inequality trivially holds.

\textbf{Case 2:} Without loss of generality, let $p \in [0,1)$ and $q = 1$, then (\ref{1.1}) yields\\
\begin{equation*}
\left|\frac{p}{8} - \frac{1}{16}\right| \le \frac{1}{8}(1 - p) + \frac{1}{16}\left(p - \frac{p}{8}\right) + \frac{1}{16}\left(1 - \frac{1}{16}\right).
\end{equation*}
Multiply the entire inequality by $256$, we obtain
\begin{align*}
16|2p - 1| &\le 32(1 - p) + 16\left(\frac{7p}{8}\right) + 16\left(\frac{15}{16}\right) \\
16|2p - 1| &\le 32 - 32p + 14p + 15 \\
16|2p - 1| &\le 47 - 18p.
\end{align*}
\begin{itemize}
    \item If $p \in \left[\frac{1}{2}, 1\right)$, then $16(2p - 1) = 32p - 16$. The inequality becomes $32p - 16 \le 47 - 18p \implies 50p \le 63$, which holds true for all $p \le 1.26$ (and thus for all $p < 1$).
    \item If $p \in \left[0, \frac{1}{2}\right)$, then $16(1 - 2p) = 16 - 32p$. The inequality becomes $16 - 32p \le 47 - 18p \implies 0 \le 31 + 14p$, which holds true for all $p \ge 0$.
\end{itemize}
Thus, (\ref{1.1}) holds for all $p, q \in [0,1]$ and  $T$ is a CRR-type mapping. Further, 
\begin{equation*}
3a + 2(b+c) = 3\left(\frac{1}{8}\right) + 2\left(\frac{1}{16} + \frac{1}{16}\right) = \frac{3}{8} + \frac{2}{8} = \frac{5}{8}<1
\end{equation*}
By Proposition \ref{prop1}, $T$ is a generalized quadrilateral CRR-type mapping.
\end{example}
\begin{example}
Let $X=\{p,q,r,s\}$. Define a metric $d$ on $M$ by the following distance matrix:
\[
\begin{array}{c|cccc}
d & p & q & r & s\\ \hline
p & 0 & 1 & 10 & 10\\
q & 1 & 0 & 10 & 10\\
r & 10 & 10 & 0 & 1\\
s & 10 & 10 & 1 & 0
\end{array}
\]
Then $(X,d)$ is a metric space. Define the mapping $T:X\to X$by
\[
T(p)=p,\qquad T(q)=q,\qquad T(r)=p,\qquad T(s)=p.
\tag{2.5}\label{2.5}
\]
\medskip
\noindent \textbf{\(T\) is a GQCRR-type mapping:} We claim that $T$ satisfies (\ref{0.1}) with \(\alpha=0\) and \(\lambda=\frac{1}{3}\). Indeed, \(3\cdot0+2\cdot\frac{1}{3}=\frac{2}{3}<1\).
Using (\ref{2.5}), we compute the left‑hand side of inequality (\ref{0.1}):
\[
\begin{aligned}
\text{LHS} &= d(Tp,Tq)+d(Tq,Tr)+d(Tr,Ts)+d(Ts,Tp) \\
&= d(p,q)+d(q,r)+d(p,p)+d(p,p) \\
&= 1 + 1 + 0 + 0 = 2.
\end{aligned}
\]
The right‑hand side (with \(\alpha=0\)) is
\[
\text{RHS} = \lambda\big(d(p,Tp)+d(q,Tq)+d(r,Tr)+d(s,Ts)\big).
\]
Using (\ref{2.5}), we have 
\[
d(p,Tp)=d(p,p)=0,\quad d(q,Tq)=d(q,q)=0,\quad d(r,Tr)=d(r,p)=10,\quad d(s,Ts)=d(s,p)=10.
\]
Thus
\[
\text{RHS} = \frac{1}{3}(0+0+10+10)=\frac{20}{3}.
\]
Since \(2 \le \frac{20}{3}\), the inequality (\ref{0.1}) holds.

\medskip
\noindent \textbf{\(T\) is not a CRR-type mapping:} Suppose for contradiction, that \(T\) satisfies (\ref{1.1}) for some \(a,b,c\ge0\) with \(a+b+c<1\). Consider the pair $(p,q)$. Since \(Tp=p\) and \(Tq=q\), we have
\[
d(Tp,Tq)=d(p,q)=1,
\]
and the right‑hand side of (\ref{1.1}) becomes
\[
a\,d(p,q) + b\,d(p,Tp) + c\,d(q,Tq)
= a\cdot 1 + b\cdot 0 + c\cdot 0 = a.
\]
Thus (\ref{1.1}) would imply
\[
1 \le a.
\]
But \(a+b+c<1\) forces \(a<1\), a contradiction.
\end{example}

\begin{example}
Let $X=[0,1]$ be equipped with the Euclidean metric $d$. Define $T:X\to X$ by
\[
T(p)=\frac{p}{2} \qquad \text{for all } p\in[0,1].
\]

\medskip
\noindent \textbf{\(T\) is a CRR-type mapping:} Clearly (\ref{1.1}) holds with \(a=\frac{1}{2}\), \(b=0\), \(c=0\), \(a+b+c=\frac{1}{2}<1\) for any \(p,q\in[0,1]\):
\[
\frac{1}{2}|p-q|\le\frac{1}{2}|p-q|+0\left|p-\frac{1}{2}p\right|+0\left|q-\frac{1}{2}q\right|.
\]

\medskip
\noindent \textbf{\(T\) is not a GQCRR-type mapping:} Take $
p=0, q=\frac{1}{3}, r=\frac{2}{3}, s=1.$
Compute the left‑hand side of inequality (\ref{0.1}):
\[
\begin{aligned}
L &= d(Tp,Tq)+d(Tq,Tr)+d(Tr,Ts)+d(Ts,Tp) \\
&= d\left(0,\frac{1}{6}\right) + d\left(\frac{1}{6},\frac{1}{3}\right) + d\left(\frac{1}{3},\frac{1}{2}\right) + d\left(\frac{1}{2},0\right) \\
&= \frac{1}{6} + \frac{1}{6} + \frac{1}{6} + \frac{1}{2}=1.
\end{aligned}
\]

Now compute the right‑hand side of (\ref{0.1}). First, the sum of pairwise distances:
\[
d(p,q)+d(q,r)+d(r,s)+d(s,p)
= \frac{1}{3}+\frac{1}{3}+\frac{1}{3}+1 = 2.
\tag{2.6}\label{2.6}
\]

Next, the sum of self‑distances:
\[
\begin{aligned}
d(p,Tp) &= d(0,0)=0,\\
d(q,Tq) &= d\left(\frac{1}{3},\frac{1}{6}\right) = \frac{1}{6},\\
d(r,Tr) &= d\left(\frac{2}{3},\frac{1}{3}\right) = \frac{1}{3},\\
d(s,Ts) &= d\left(1,\frac{1}{2}\right) = \frac{1}{2}.
\end{aligned}
\]
Thus
\[
d(p,Tp)+d(q,Tq)+d(r,Tr)+d(s,Ts)
= 0 + \frac{1}{6} + \frac{1}{3} + \frac{1}{2} = 1.
\tag{2.7}\label{2.7}
\]

Using (\ref{2.6}) and (\ref{2.7}), the right‑hand side of (\ref{0.1}) becomes
\[
R = \alpha \cdot 2 + \lambda \cdot 1 = 2\alpha + \lambda.
\]

For (\ref{0.1}) to hold, we would need
\[
L \le R \quad\Longrightarrow\quad 1 \le 2\alpha + \lambda.\tag{2.8}\label{2.8}
\]

However, the GQCRR coefficient condition requires
\[
3\alpha + 2\lambda < 1.
\tag{2.9}\label{2.9}
\]
We claim that (\ref{2.9}) implies \(2\alpha+\lambda < 1\). Indeed, from (\ref{2.9}) we have
\[
\alpha < \frac{1-2\lambda}{3} \quad (\text{since } \alpha\ge0).
\]
Then
\[
2\alpha+\lambda < \frac{2(1-2\lambda)}{3} + \lambda
= \frac{2-4\lambda+3\lambda}{3} = \frac{2-\lambda}{3} \le \frac{2}{3} < 1,
\]
because \(\lambda\ge0\). Hence \(2\alpha+\lambda < 1\), which contradicts (\ref{2.8}).
\end{example}

The preceding two examples demonstrate that the classes of CRR-type mappings and generalized quadrilateral CRR-type mappings are independent of one another.
\begin{example}
 Let \(X = \{p,q,r,s\}\) be a metric space with metric $d$ defined as follows:
\[
\begin{aligned}
&d(p,q)=1,\quad d(q,r)=1,\quad d(r,p)=1,\\
&d(q,s)=4,\quad d(s,p)=4,\quad d(r,s)=4.
\end{aligned}
\tag{2.10}\label{2.10}
\]
Define the mapping \(T:X\to X\) by
\[
Tp=p,\qquad Tq=q,\qquad Tr=q,\qquad Ts=q.
\tag{2.11}\label{2.11}
\]

\medskip
\noindent \textbf{\(T\) is a GQCRR-type mapping:} Choose \(\alpha = \frac{1}{5}\) and \(\lambda = \frac{1}{10}\). Then
\[
3\alpha + 2\lambda = 3\left(\frac{1}{5}\right) + 2\left(\frac{1}{10}\right) = \frac{3}{5} + \frac{1}{5} = \frac{4}{5} < 1,
\]

Now, we verify the inequality (\ref{0.1}) for \(\{p,q,r,s\}\). The left‑hand side is
\[
\begin{aligned}
\text{LHS} &= d(Tp,Tq) + d(Tq,Tr) + d(Tr,Ts) + d(Ts,Tp) \\
&= d(p,q) + d(q,q) + d(q,q) + d(q,p) \\
&= 1 + 0 + 0 + 1 = 2.
\end{aligned}
\]
The right‑hand side is
\[
\begin{aligned}
\text{RHS} &= \alpha\big(d(p,q)+d(q,r)+d(r,s)+d(s,p)\big) \\
&\quad + \lambda\big(d(p,Tp)+d(q,Tq)+d(r,Tr)+d(s,Ts)\big).
\end{aligned}
\]
Using (\ref{2.10}) and (\ref{2.11}), we have
\[
d(p,q)=1,\quad d(q,r)=1,\quad d(r,s)=4,\quad d(s,p)=4,
\]
Also,
\[
d(p,Tp)=d(p,p)=0,\quad d(q,Tq)=d(q,q)=0,\quad d(r,Tr)=d(r,q)=1,\quad d(s,Ts)=d(s,q)=4.
\]
Therefore,
\[
\text{RHS} = \frac{1}{5}(10) + \frac{1}{10}(5) = 2 + 0.5 = 2.5.
\]
Since \(2 \le 2.5\), the inequality (\ref{0.1}) holds. Thus \(T\) is a GQCRR-type mapping.

\medskip
\noindent \textbf{\(T\) is not a GCRR-type mapping:}
Suppose for contradiction, that \(T\) is a GCRR-type mapping on \(\{p,q,r\}\). Then there exist \(\alpha',\lambda'\ge0\) with \(2\alpha'+\frac{3}{2}\lambda'<1\) such that 
\[
d(Tp,Tq)+d(Tq,Tr)+d(Tr,Tp)
\le \alpha'\big(d(p,q)+d(q,r)+d(r,p)\big)
+ \lambda'\big(d(p,Tp)+d(q,Tq)+d(r,Tr)\big).
\tag{2.12}\label{2.12}
\]
Using (\ref{2.11}), we have
\[
\begin{aligned}
d(Tp,Tq) &= d(p,q)=1,\\
d(Tq,Tr) &= d(q,q)=0,\\
d(Tr,Tp) &= d(q,p)=1.
\end{aligned}
\]
So the left‑hand side of (\ref{2.12}) is
\[
1+0+1=2.
\]
Using \(d(p,q)=d(q,r)=d(r,p)=1\) and $d(p,Tp)=d(p,p)=0, d(q,Tq)=d(q,q)=0, d(r,Tr)=d(r,q)=1,$ the right-hand side of (\ref{2.12}) becomes $3\alpha' + \lambda'$. To hold the inequality (\ref{2.12}), we need $2 \le 3\alpha' + \lambda'$, which is not possible. Hence \(T\) is not a GCRR-type mapping.
\end{example}

\begin{proposition}\label{prop2}
Let $(X,d)$ be a metric space with at least four points, and let $T:X\to X$ be a GQCRR-type mapping with some $\alpha,\lambda \ge 0$ satisfying $3\alpha + 2\lambda < 1$. If $p$ is an accumulation point of $X$ and $T$ is continuous at $p$, then the inequality
\[
d(Tp, Tq) \le \alpha\, d(p, q) + \frac{3\lambda}{2}\, d(p, Tp) + \frac{\lambda}{2}\, d(q, Tq)
\tag{2.13}\label{2.13}
\]
holds for every $q \in X$.
\end{proposition}
\begin{proof}
Let $p$ be an accumulation point, and let $q\in X$. If $p = q$, then the inequality (\ref{2.13}) is trivial. Assume $p \ne q$.
Since $p$ is an accumulation point, there exist two sequences $\{r_n\}$ and $\{s_n\}$ in $X$ such that $r_n \to p$ and $s_n \to p$ as $n \to \infty$ and the points $p, q, r_n, s_n$ are pairwise distinct for every $n\in\mathbb{N}$. Apply the inequality (\ref{0.1}) to these four points $p, q, r_n$ and $s_n$, we obtain
\[
\begin{aligned}
&d(Tp, Tq) + d(Tq, Tr_n) + d(Tr_n, Ts_n) + d(Ts_n, Tp) \\
&\quad \le \alpha \Big(d(p, q) + d(q, r_n) + d(r_n, s_n) + d(s_n, p)\Big) \\
&\quad \quad + \lambda \Big(d(p, Tp) + d(q, Tq) + d(r_n, Tr_n) + d(s_n, Ts_n)\Big).
\end{aligned}
\tag{2.14}\label{2.14}
\]
Since $r_n \to p$, $s_n \to p$, and $T$ is continuous at $p$, we have $Tr_n \to Tp$ and $Ts_n \to Tp$.
Therefore, using the continuity of metric, we have $d(r_n, Tr_n) \rightarrow d(p, Tp)$ and $d(s_n, Ts_n) \rightarrow d(p, Tp)$.
Taking $n \to \infty$ on both sides of (\ref{2.14}), we obtain
\[
d(Tp, Tq) + d(Tq, Tp)\le \alpha \big(d(p, q) + d(q, p)\big) + \lambda \big(d(p, Tp) + d(q, Tq) + d(p, Tp) + d(p, Tp)\big),
\]
which is equivalent to the desired inequality (\ref{2.13}).
\end{proof}
\begin{corollary}
Let $(X,d)$ be a metric space with atleast four points, and let $T:X\to X$ be a continuous GQCRR-type mapping with some $\alpha,\lambda \ge 0$ and assume that every point of $X$ is an accumulation point. Then $T$ is a CRR-type mapping.\end{corollary}
\begin{proof}
Since $T$ is continuous everywhere and every point of $X$ is an accumulation point, we may apply Proposition \ref{prop2} to any $p,q\in X$. Interchanging $p$ and $q$ in (\ref{2.13}), we obtain:

\[
d(Tq,Tp) \le \alpha\, d(q,p) + \frac{3\lambda}{2}\, d(q,Tq) + \frac{\lambda}{2}\, d(p,Tp).
\tag{2.15}\label{2.15}
\]
 Adding (\ref{2.13}) and (\ref{2.15}) and dividing by $2$, we obtain
\[
d(Tp,Tq) \le \alpha\,d(p,q) + \lambda\,d(p,Tp) + \lambda\,d(q,Tq).
\]
Since $3\alpha+2\lambda<1$, $\alpha\ge0$, we have

\[
\alpha + 2\lambda < \alpha + (1-3\alpha) = 1 - 2\alpha \le 1.
\]
Thus $T$ is a CRR-type mapping. 
\end{proof}
\section{Main Results}

\begin{theorem}\label{theorem3.1}
Let $(X, d)$ be a complete metric space with $|X| \ge 4$, and let $T: X \to X$ be a self-mapping satisfying the following two conditions:
\begin{enumerate}
    \item[(i)] $T$ does not possess periodic points of prime period $2$ and prime period $3$, i.e., $T^2 p \ne p$ and $T^3 p \ne p$ for all $p \in X$ such that $Tp \ne p$.
     \item[(ii)]$T$ is a generalized quadrilateral CRR-type mapping on $X$.
    \end{enumerate}
    Then $T$ has a fixed point in $X$. Moreover, the number of fixed points is at most three.
\end{theorem}
\begin{proof}
Let $p_0 \in X$ be an arbitrary starting point. Define the Picard iteration sequence $\{p_n\}$ by
\[
p_{n+1} = T p_n, \quad \forall~n \ge 0.
\]
If $p_n$ is a fixed point of $T$, then we are done. Suppose $p_n$ is not a fixed point of $T$ for any $n \ge 0$. Since $p_{n+1} = T p_n\implies p_n\ne Tp_n=p_{n+1}$ for all $n \ge 0$. Since $T$ has no periodic points of prime period 2, $p_n\ne T^2 p_n=p_{n+2}$. Since it has no periodic points of prime period 3, $p_n\ne T^3 p_n=p_{n+3}$.
Thus, for every $n \ge 0$, the four consecutive points $p_n, p_{n+1}, p_{n+2}, p_{n+3}$ are pairwise distinct. Setting $p = p_n, q = p_{n+1}, r = p_{n+2}, s = p_{n+3}$ and substituting $p_{k+1} = T p_k$ in (\ref{0.1}) yields:
\begin{align*}
    &d(p_{n+1}, p_{n+2}) + d(p_{n+2}, p_{n+3}) + d(p_{n+3}, p_{n+4}) + d(p_{n+4}, p_{n+1}) \\
    &\quad\le \alpha \big( d(p_n, p_{n+1}) + d(p_{n+1}, p_{n+2}) + d(p_{n+2}, p_{n+3}) + d(p_{n+3}, p_n) \big) \\
    &\quad\quad + \lambda \big( d(p_n, p_{n+1}) + d(p_{n+1}, p_{n+2}) + d(p_{n+2}, p_{n+3}) + d(p_{n+3}, p_{n+4}) \big).
\end{align*}

Using the triangle inequalities 
\[
d(p_{n+1}, p_{n+2}) + d(p_{n+2}, p_{n+3}) + d(p_{n+4}, p_{n+1}) \ge d(p_{n+3}, p_{n+4})
\]
and
\[
d(p_{n+3}, p_n) \le d(p_n, p_{n+1}) + d(p_{n+1}, p_{n+2}) + d(p_{n+2}, p_{n+3}),
\]
we get
\begin{align*}
    2 d(p_{n+3}, p_{n+4}) &\le 2\alpha \big( d(p_n, p_{n+1}) + d(p_{n+1}, p_{n+2}) + d(p_{n+2}, p_{n+3}) \big) \\
    &\quad + \lambda \big( d(p_n, p_{n+1}) + d(p_{n+1}, p_{n+2}) + d(p_{n+2}, p_{n+3}) + d(p_{n+3}, p_{n+4}) \big).
\end{align*}
Rearranging of terms yields
\[
(2 - \lambda) d(p_{n+3}, p_{n+4}) \le (2\alpha + \lambda) \big( d(p_n, p_{n+1}) + d(p_{n+1}, p_{n+2}) + d(p_{n+2}, p_{n+3}) \big).
\]
Since $3\alpha + 2\lambda < 1 \implies \lambda < 1/2$, we have $2 - \lambda > 0$. 
\[
d(p_{n+3}, p_{n+4}) \le \frac{3(2\alpha + \lambda)}{2 - \lambda} \max \Big\{ d(p_n, p_{n+1}), d(p_{n+1}, p_{n+2}), d(p_{n+2}, p_{n+3}) \Big\}.
\]

Let $\gamma = \frac{3(2\alpha + \lambda)}{2 - \lambda}$, using the condition $3\alpha + 2\lambda < 1$, we have $\gamma \in [0, 1)$. Setting $a_n = d(p_n, p_{n+1})$ for $n \ge 0$ and $A = \max \{a_0, a_1, a_2\}$, by induction for $n \ge 3$, we have
\[
a_n \le \gamma^{\lfloor \frac{n}{3} \rfloor} A \le \gamma^{\frac{n}{3} - 1} A.\tag{3.1}\label{0.3}
\]
For any $m \ge 1$ and $n \ge 3$,
\[
d(p_n, p_{n+m}) \le \sum_{j=0}^{m-1} a_{n+j} \le A \sum_{j=0}^{m-1} \gamma^{\frac{n+j}{3} - 1} \le \frac{A \gamma^{\frac{n-3}{3}}}{1 - \gamma^{1/3}}.
\]
Since $0 \le \gamma < 1$, as $n \to \infty$, $d(p_n, p_{n+m}) \to 0$ for all $m \ge 1$. Hence, $\{p_n\}$ is a Cauchy sequence. By completeness of $(X, d)$, there exists $p^* \in X$ such that $ \lim \limits_{n \to \infty} p_n = p^*$.

If $p^* = p_k$ for some $k \ge 0$, the Cauchy property prevents $p^*$ from repeating at any $m > k$ (otherwise $\{p_n\}$ would cycle with positive distance). Thus, $p^* = p_k$ for at most one index $k$. Consequently, for all $n > k$, the four points $p^*, p_n, p_{n+1}, p_{n+2}$ are pairwise distinct.

Applying the inequality (\ref{0.1}) to $p = p^*, q = p_n, r = p_{n+1}, s = p_{n+2}$ yields:
\begin{align*}
    &d(T p^*, p_{n+1}) + d(p_{n+1}, p_{n+2}) + d(p_{n+2}, p_{n+3}) + d(p_{n+3}, T p^*) \\
    &\quad\le \alpha \big( d(p^*, p_n) + d(p_n, p_{n+1}) + d(p_{n+1}, p_{n+2}) + d(p_{n+2}, p^*) \big) \\
    &\quad \quad+ \lambda \big( d(p^*, T p^*) + d(p_n, p_{n+1}) + d(p_{n+1}, p_{n+2}) + d(p_{n+2}, p_{n+3}) \big).
\end{align*}
Using $d(T p^*, p_{n+1}) \ge d(p^*, T p^*) - d(p^*, p_{n+1})$, we obtain
\begin{align*}
    (1 - \lambda) d(p^*, T p^*) &\le (1 + \alpha) d(p^*, p_{n+1}) + \alpha d(p_{n+2}, p^*) + \alpha d(p^*, p_n) \\
    &\quad + (\alpha + \lambda) d(p_n, p_{n+1}) + (\alpha + \lambda - 1) d(p_{n+1}, p_{n+2}) + (\lambda - 1) d(p_{n+2}, p_{n+3}).
\end{align*}
Taking the limit as $n \to \infty$, we get
\[
(1 - \lambda) d(p^*, T p^*) \le 0.
\]
Since $\lambda < 1/2$, we conclude that $d(p^*, T p^*) = 0$, so $T p^* = p^*$.

Finally, suppose $T$ has four distinct fixed points $u_1, u_2, u_3, u_4 \in X$. Substituting these into the inequality (\ref{0.1}) gives
\[
d(u_1, u_2)+d(u_2, u_3)+d(u_3, u_4)+d(u_4, u_1) \le \alpha d(u_1, u_2)+d(u_2, u_3)+d(u_3, u_4)+d(u_4, u_1) +\lambda(0),
\]
which implies $1 \le \alpha$, a contradiction. Therefore, $T$ has at most three fixed points.
\end{proof}

\begin{example}
Let $X=\{a,b,c,d\}$ and define a metric $d:X\times X\to[0,\infty)$ by
\[
d(p,q)=
\begin{cases}
0, & p=q,\\
1, & \{p,q\}\in\bigl\{\{a,b\},\{a,c\},\{b,c\}\bigr\},\\
8, & \{p,q\}=\{a,d\},\\
7, & \{p,q\}\in\bigl\{\{b,d\},\{c,d\}\bigr\}.
\end{cases}
\]
Then $(X, d)$ is a complete metric space.\\
Define $T:X\to X$ by
\[
T(a)=a,\qquad
T(b)=b,\qquad
T(c)=c,\qquad
T(d)=a.
\]
Then $T$ is not a CRR-type mapping. If $T$ is a generalized CRR-type mapping \cite{BishtPetrovCRR}, then we have
\begin{align*}
d(a,b) + d(b,c) + d(c,a)&\le \alpha\bigl(d(a,b) + d(b,c) + d(c,a)\bigr) + \lambda\bigl(d(a,a) + d(b,b) + d(c,c)\bigr)\\
3&\le 3\alpha,
\end{align*} 
which is a contradiction to (\ref{1.5}). Thus $T$ is not a generalized CRR-type mapping. However, it satisfies all the hypotheses of Theorem \ref{theorem3.1} with $\alpha=\frac{1}{10}$ and $\lambda=\frac{3}{10}$. Consequently, by Theorem \ref{theorem3.1}, $T$ has three fixed point in $X$. In fact, $\operatorname{Fix}(T)=\{a,b,c\}$.
\end{example}

\begin{example}
Let $X = \{0, 1, 2\} \cup [100, 200]$ be equipped with the usual metric $d$. Then $(X,d)$ is a complete metric space. Define \(T: X \to X\) by
\[
T(0)=0, T(1)=1, T(2)=2 ~and~T(x)=0~\forall~ x \in [100,200].
\]

For any $x\in [100,200], T(x)=0, T^2(x)=T(0)=0$ and $T^3(x)=0$. Therefore, $T$ has no periodic point of prime period $2$ and $3$. Also, $T$ satisfies the GQCRR inequality (\ref{0.1}) with $\alpha = \frac{1}{10}$ and $ \lambda = \frac{3}{10}$. Therefore, Applying Theorem \ref{theorem3.1}, we conclude that $T$ has exactly three fixed points namely $0, 1, 2$.
\end{example}

\begin{remark}
Suppose that under the hypotheses of Theorem \ref{theorem3.1}, the mapping $T$ has a fixed point $p^*$, which is the limit of some iterative sequence $p_0,\quad p_1 = Tp_0,\quad p_2 = Tp_1,\quad \ldots $ such that $p_n \neq p^*$ for all $n\ge 1$. Then $p^*$ is the unique fixed point.

Indeed, assume that $T$ admits another fixed point $q^* \neq p^*$. Since $p_n \neq p^*$ for all $n$, and $q^* \neq p^*$, it is clear that the points $p_n$ are distinct from both $p^*$ and $q^*$ for all $n\ge1$. In particular, for any $n\ge1$, the four points $p^*,  q^*, p_n,  p_{n+1}$ are pairwise distinct.

Applying the inequality (\ref{0.1}) to these four points, we get

\[
\begin{aligned}
&d(Tp^*, Tq^*) + d(Tq^*, Tp_n) + d(Tp_n, Tp_{n+1}) + d(Tp_{n+1}, Tp^*) \\
&\quad \le \alpha\big(d(p^*, q^*) + d(q^*, p_n) + d(p_n, p_{n+1}) + d(p_{n+1}, p^*)\big) \\
&\quad \quad + \lambda\big(d(p^*, Tp^*) + d(q^*, Tq^*) + d(p_n, Tp_n) + d(p_{n+1}, Tp_{n+1})\big).
\end{aligned}
\tag{3.2}\label{3.1}
\]

Since $p^*$ and $q^*$ are fixed points, we have $Tp^* = p^*$ and $Tq^* = q^*$. Also $Tp_n = p_{n+1}$ and $Tp_{n+1} = p_{n+2}$. Thus (\ref{3.1}) reduces to

\[
\begin{aligned}
&d(p^*, q^*) + d(q^*, p_{n+1}) + d(p_{n+1}, p_{n+2}) + d(p_{n+2}, p^*) \\
&\quad \le \alpha\big(d(p^*, q^*) + d(q^*, p_n) + d(p_n, p_{n+1}) + d(p_{n+1}, p^*)\big) \\
&\quad \quad + \lambda\big(0 + 0 + d(p_n, p_{n+1}) + d(p_{n+1}, p_{n+2})\big).
\end{aligned}
\tag{3.3}\label{3.2}
\]

Now, let $n \to \infty$. Since $p_n \to p^*$, we have

\[
d(q^*, p_n) \to d(q^*, p^*), \quad d(p_{n+1}, p^*) \to 0, \quad d(p_{n+2}, p^*) \to 0,
\]

and
\[
d(p_n, p_{n+1}) \to 0, \quad d(p_{n+1}, p_{n+2}) \to 0.
\]

Therefore, inequality (\ref{3.2}) yields
\[
2d(p^*, q^*) \le 2\alpha d(p^*, q^*).
\]

Since $\alpha < \frac{1}{3} < 1$, we have $2 - 2\alpha > 0$, so the inequality implies $d(p^*, q^*) = 0$, a contradiction to $p^* \neq q^*$. Hence $T$ cannot have another fixed point. Thus $p^*$ is the unique fixed point of $T$.
\end{remark}

\begin{proposition}
Let $(X,d)$ be a metric space with $\left|X\right|\ge4$, and let $T: X\to X$ be a generalized quadrilateral CRR-type mapping on $X$ such that $T p^* = p^*$. Then \(T\) is continuous at \(p^*\).
\end{proposition}

\begin{proof}
Let \(\{p_n\}\) be a sequence in \(X\) such that \(p_n \to p^*\). We must show \(T p_n \to T p^* = p^*\). Define $u_n := d(Tp_n, p^*)~\text {and}~a_n := d(p_n, p^*)$, by the convergence of $\{p_n\}$, we have $a_n \to 0$ as $n\to\infty$. If infinitely many terms of $\{p_n\}$ equal \(p^*\), then continuity is trivial. Otherwise, discarding finitely many terms, we may assume \(p_n \ne p^*\). For sufficiently large \(n\), the four points $p^*,\; p_n,\; p_{n+1},\; p_{n+2}$
are pairwise distinct. 

Applying the inequality (\ref{0.1}) to the four points with $p=p^*, q= p_n, r= p_{n+1}, s=p_{n+2}$, we obtain
\[
\begin{aligned}
&d(p^*,Tp_n) + d(Tp_n,Tp_{n+1}) + d(Tp_{n+1},Tp_{n+2}) + d(Tp_{n+2},Tp^*) \\
&\quad \le \alpha\big(d(p^*,p_n)+d(p_n,p_{n+1})+d(p_{n+1},p_{n+2})+d(p_{n+2},p^*)\big) \\
&\quad \quad + \lambda\big(d(p^*,Tp^*)+d(p_n,Tp_n)+d(p_{n+1},Tp_{n+1})+d(p_{n+2},Tp_{n+2})\big).
\end{aligned}
\tag{3.4}\label{3.3}
\]
Since \(Tp^*=p^*\), the left‑hand side of (\ref{3.3}) becomes
\[
u_n + d(Tp_n,Tp_{n+1}) + d(Tp_{n+1},Tp_{n+2}) + u_{n+2}.
\]
Because distances are non-negative, this is at least $u_n + u_{n+2}$.\\
Using the triangle inequality
\[
d(p_n,p_{n+1}) \le d(p_n,p^*) + d(p^*, p_{n+1}) = a_n + a_{n+1},
\]
\[
d(p_{n+1},p_{n+2}) \le a_{n+1} + a_{n+2},
\]
and
\[
d(p_k,Tp_k) \le d(p_k,p^*) + d(p^*,Tp_k) = a_k + u_k \qquad (k=n,n+1,n+2).
\]

Substituting these estimates with \(d(p^*,Tp^*)=0\) into (\ref{3.3}) yields
\[
u_n + u_{n+2}
\le \alpha(2a_n + 2a_{n+1} + 2a_{n+2})
+ \lambda(a_n + a_{n+1} + a_{n+2} + u_n + u_{n+1} + u_{n+2}).
\]
Define
\[
\epsilon_n := \alpha(2a_n + 2a_{n+1} + 2a_{n+2}) + \lambda(a_n + a_{n+1} + a_{n+2}).
\]
Thus the inequality becomes
\[
u_n + u_{n+2} \le \epsilon_n + \lambda(u_n + u_{n+1} + u_{n+2}).
\]
Rearranging terms,
\[
(1-\lambda)(u_n + u_{n+2}) \le \epsilon_n + \lambda u_{n+1}. \tag{3.5}\label{3.4}
\]

We now take the limit superior as \(n\to\infty\) on both sides of (\ref{3.4}).  
Let
\[
L := \limsup_{n\to\infty} u_n.
\]
Since \(u_n\ge 0\), we have \(L\ge 0\).  
Also,
\[
\limsup_{n\to\infty} (u_n+u_{n+2}) = 2L ,\quad \limsup_{n\to\infty} u_{n+1} = L,
\]
and
\[
\limsup_{n\to\infty} \epsilon_n = 0.
\]

Therefore, applying \(\limsup\) to (\ref{3.4}) gives
\[
(2-3\lambda)L \le 0. 
\]
Since $\lambda < \frac{1}{2}.$ Consequently,
\[
2 - 3\lambda > 2 - \frac32 = \frac12 > 0.
\]
Therefore,
\[
\limsup_{n\to\infty} d(Tp_n, p^*) = 0.
\]
Since $d(Tp_n,p^*)\ge 0$ for all $n\implies \lim \limits_{n\to\infty} d(Tp_n, p^*) = 0,$ $i.e., T p_n \rightarrow p^* = T p^*$.
\end{proof}

\begin{definition}
Let $(X,d)$ be a metric space. A mapping $T:X\to X$ is said to be \emph{asymptotically regular}~\cite{asym}
if, for every $p\in X$,
\[
\lim_{n\to\infty}d(T^n p,T^{n+1}p)=0.
\]
\end{definition}

Asymptotic regularity broadens the applicability of contractive mappings by relaxing parameter restrictions and extending the class of mappings for which fixed-point theorems can be established.
\begin{remark}\label{4.2}
Let $(X,d)$ be a metric space and $T: X \to X$ be a self-mapping. Consider the sequence of iterates $\{p_n\}$ defined by $p_{n+1} = T p_n$ from some initial point $p_0$ for all $n \ge 0$. If $T$ is asymptotically regular and the sequence $\{p_n\}$ contains no fixed point of $T$, then every term of the sequence is distinct i.e., $p_i \ne p_j$, $i \ne j$. Indeed, if $p_i = p_j$ for some $i < j$, then the sequence becomes eventually periodic from index $i$ onward, which forces $d(p_n, p_{n+1})$ to fail to converge to zero, thereby contradicting the asymptotic regularity of $T$. 
\end{remark}
\begin{theorem}
Let $(X,d)$ be a complete metric space with $\left|X\right|\ge4$, and let $T:X\to X$ be an asymptotically regular mapping that satisfies the GQCRR inequality with coefficients $\alpha\in[0,\frac{1}{3})$ and $\lambda\in[0,1)$. Then $T$ has a fixed point. Moreover, $T$ can have at most three fixed points.
\end{theorem}

\begin{proof}
Consider an arbitrary point $p_0 \in X$ and define the Picard sequence as
$$T(p_0) = p_1,\quad T(p_1) = p_2,\quad\dots,\quad p_{n+1}=Tp_n.$$ Suppose that $\{p_n\}$ does not contain a fixed point of $T$. we now prove that $\{p_n\}$ is a Cauchy sequence. It is sufficient to show that $d(p_n, p_{n+m}) \to 0$ as $n \to \infty$ for all $m> 0$. If $m = 1$, then this follows from the definition of asymptotic regularity. Let $m\ge 2$. By Remark \ref{4.2}, the points $p_n$, $p_{n+m-1}$, $p_{n+m}$ and $p_{n+1}$ are pairwise distinct.

We have the triangle inequality, 

\begin{equation}
d(p_n, p_{n+m}) \le d(p_n, p_{n+1}) + d(p_{n+1}, p_{n+m}).
\tag{3.6}\label{4.1}
\end{equation}

Now, apply the inequality (\ref{0.1}) to the four distinct points $p_n, p_{n+m-1}, p_{n+m}, p_{n+1}$ and Substitute $Tp_k = p_{k+1}$, we obtain
\begin{equation}
\begin{aligned}
&d(p_{n+1}, p_{n+m}) + d(p_{n+m}, p_{n+m+1}) + d(p_{n+m+1}, p_{n+2}) + d(p_{n+2}, p_{n+1}) \\
&\quad \le \alpha\big(d(p_n, p_{n+m-1}) + d(p_{n+m-1}, p_{n+m}) + d(p_{n+m}, p_{n+1}) + d(p_{n+1}, p_n)\big) \\
&\quad \quad + \lambda\big(d(p_n, p_{n+1}) + d(p_{n+m-1}, p_{n+m}) + d(p_{n+m}, p_{n+m+1}) + d(p_{n+1}, p_{n+2})\big).
\end{aligned}
\tag{3.7}\label{4.2}
\end{equation}

From (\ref{4.2}), we can isolate the term $d(p_{n+1}, p_{n+m})$ by noting that all remaining terms on the left side are non-negative. Hence,

\begin{equation}
\begin{aligned}
d(p_{n+1}, p_{n+m}) &\le \alpha\big(d(p_n, p_{n+m-1}) + d(p_{n+m-1}, p_{n+m}) + d(p_{n+m}, p_{n+1}) + d(p_{n+1}, p_n)\big) \\
&\quad + \lambda\big(d(p_n, p_{n+1}) + d(p_{n+m-1}, p_{n+m}) + d(p_{n+m}, p_{n+m+1}) + d(p_{n+1}, p_{n+2})\big).
\end{aligned}
\tag{3.8}\label{4.3}
\end{equation}

Now, using the triangle inequality,

\begin{equation}
d(p_n, p_{n+m-1}) \le d(p_n, p_{n+m}) + d(p_{n+m}, p_{n+m-1}),
\tag{3.9}\label{4.4}
\end{equation}

and

\begin{equation}
d(p_{n+m}, p_{n+1}) \le d(p_{n+m}, p_n) + d(p_n, p_{n+1})
\tag{3.10}\label{4.5}
\end{equation}
Substituting (\ref{4.4}) and (\ref{4.5}) into (\ref{4.3}), we get
\begin{equation}
\begin{aligned}
d(p_{n+1}, p_{n+m}) &\le 2\alpha d(p_n, p_{n+m}) + 2\alpha d(p_{n+m}, p_{n+m-1}) + (2\alpha + \lambda)d(p_n, p_{n+1}) \\
&\quad + \lambda d(p_{n+m-1}, p_{n+m}) + \lambda d(p_{n+m}, p_{n+m+1}) + \lambda d(p_{n+1}, p_{n+2}).
\end{aligned}
\tag{3.11}\label{4.6}
\end{equation}
Using (\ref{4.6}) into (\ref{4.1}) and after simplifying the terms, we obtain
\begin{equation*}
\begin{aligned}
d(p_n, p_{n+m}) &\le \frac{1}{1 - 2\alpha}\big[(1 + 2\alpha + \lambda)d(p_n, p_{n+1}) + (2\alpha + \lambda)d(p_{n+m}, p_{n+m-1}) \\
&\quad + \lambda d(p_{n+m}, p_{n+m+1}) + \lambda d(p_{n+1}, p_{n+2})\big]\to 0,
\end{aligned}
\end{equation*}
as $n \to \infty$. Therefore, $\{p_n\}$ is a Cauchy sequence. The rest of the proof proceeds analogously to that of Theorem \ref{theorem3.1}.
\end{proof}

\section{Fixed Point Theorems in Incomplete Metric Spaces}
We now present the following two fixed‑point theorems that do not require completeness. Theese theorems are analogue of Theorem 1 and Theorem 2 from \cite{kannan}.
\begin{theorem}\label{theorem5.1}
Let $(X,d)$ be a metric space with $|X|\geq 4$, and let
$T:X\to X$ be a generalized quadrilateral CRR-type mapping. Assume that the following conditions hold:
\begin{enumerate}
    \item[(i)] $T$ has no periodic point of prime period $2$ or $3$;
  \item[(ii)] $T$ is continuous at $p^*$;
    \item[(iii)] There exists $p_0\in X$ such that, for the Picard
    iteration $p_{n+1}=Tp_n,~ n=0,1,2,\ldots,$ the sequence $\{p_n\}$ possesses a subsequence $\{p_{n_k}\}$
    converging to some $p^*\in X$.
\end{enumerate}
Then $p^*$ is a fixed point of $T$. The number of fixed points is at most three.
\end{theorem}

\begin{proof}
Let $p_{n+1}=Tp_n,~ n=0,1,2,\ldots,$ where $p_0\in X$ is the point specified in condition $(iii)$. If $p_n=p_{n+1}$ for some $n\geq0$, then $p_n$ is a fixed point of $T$. In this case, the conclusion follows immediately. Hence, we suppose that $p_n\neq p_{n+1} ~\text{for every }n\geq0.$

Since $T$ has no periodic points of prime period $2$ or $3$, we also
have $p_n\neq p_{n+2}~\text{and}~p_n\neq p_{n+3} $ for every $n\geq0$. Consequently, the four consecutive points
$p_n, p_{n+1}, p_{n+2}, p_{n+3}$ are pairwise distinct for every $n\geq0$.

By assumption, there exists a subsequence $p_{n_k}\rightarrow p^*$. Since $T$ is continuous at $p^*$, we have $Tp_{n_k}=p_{n_k+1}\rightarrow Tp^*$. Suppose $Tp^*\neq p^*$.
Choose $r>0$ such that $0<r<\frac13d(p^*,Tp^*).$ Consequently, there exists $K\in\mathbb{N}$ such that, for every $k\geq K$,
\[
d(p_{n_k},p^*)<r
\]
and
\[
d(p_{n_k+1},Tp^*)<r.
\]
Hence,
\[
\begin{aligned}
d(p_{n_k},p_{n_k+1}) >r.
\end{aligned}
\tag{5.1}\label{5.1}
\]

On the other hand, from (\ref{0.3}),
\[
d(p_{n_k},p_{n_k+1})
=a_{n_k}
\rightarrow0,
\]
which contradicts (\ref{5.1}). Hence, $Tp^*=p^*.$

As established in the proof of Theorem \ref{theorem3.1}, there can be at most three fixed points.
\end{proof}
\begin{theorem}\label{theorem5.2}
Let $(X,d)$ be a metric space with $|X|\ge 4$, and let $T:X\to X$ be continuous on $X$ and satisfies the GQCRR condition on a dense subset $Y$ i.e., the inequality (\ref{0.1}) holds for all pairwise distinct points $p,q,r,s\in Y.$ Suppose that:

\begin{enumerate}
 \item[(i)] $T$ has no periodic points of prime period $2$ and $3$;
 \item[(ii)] There exists $p_0\in X$ such that, for the Picard
    iteration $p_{n+1}=Tp_n,~ n=0,1,2,\ldots,$ the sequence $\{p_n\}$ possesses a subsequence $\{p_{n_k}\}$
    converging to some $p^*\in X$.
\end{enumerate}

Then $p^*$ is a fixed point of $T$. The number of fixed points is at most three.
\end{theorem}
\begin{proof}
\textbf{Case 1:}
Without loss of generality, suppose $p,q,r\in Y$ and $s\in X\setminus Y$. Since $Y$ is dense in $X$, there exists a sequence $\{s_n\}\subset Y$ such that $s_n\to s$. we may assume that all points $p,q,r,s_n$ are pairwise distinct. Applying (\ref{0.1}) to the points $p,q,r,s_n$ gives
\[
\begin{aligned}
&d(Tp,Tq) + d(Tq,Tr) + d(Tr,Ts_n) + d(Ts_n,Tp) \\
&\quad \le \alpha\big(d(p,q)+d(q,r)+d(r,s_n)+d(s_n,p)\big) \\
&\quad \quad + \lambda\big(d(p,Tp)+d(q,Tq)+d(r,Tr)+d(s_n,Ts_n)\big).
\end{aligned} 
\]
Since $T$ is continuous on $M$, we have $Ts_n\to Ts$. Also, by continuity of the metric $d$, we have
\[
d(r,s_n)\to d(r,s)~\text{and}~ d(s_n,p)\to d(s,p).
\]
using the triangle inequality,
\[
d(s_n,Ts_n) \le d(s_n,s) + d(s,Ts) + d(Ts,Ts_n),
\]
we obtain
\begin{align*}
&d(Tp,Tq) + d(Tq,Tr) + d(Tr,Ts) + d(Ts,Tp) \\
&\quad \le \alpha\big(d(p,q)+d(q,r)+d(r,s)+d(s,p)\big) \\
&\quad \quad + \lambda\big(d(p,Tp)+d(q,Tq)+d(r,Tr)+d(s_n,s) + d(s,Ts) + d(Ts,Ts_n)\big).
\end{align*}
Thus, taking the limit $n\to\infty$, we obtain $d(s_n,s)\rightarrow 0$ and $d(Ts,Ts_n)\rightarrow 0$. So (\ref{0.1}) holds.

\textbf{Case 2:}
Suppose $p,q\in Y$ and $r,s\in X\setminus Y$. Choose sequences $\{r_n\},\{s_n\}\subset Y$ such that $r_n\to r$ and $s_n\to s$, with all points $p,q,r_n,s_n$ pairwise distinct. Applying the same limiting argument as in case 1, we obtain the inequality (\ref{0.1}).

\textbf{Case 3:}
Suppose $p\in Y$ and $q,r,s\in X\setminus Y$. Choose sequences $\{q_n\},\{r_n\},\{s_n\}\subset Y$ such that $q_n\to q$, $r_n\to r$, $s_n\to s$, with all points $p,q_n,r_n,s_n$ pairwise distinct. Applying the limiting argument, we obtain the inequality (\ref{0.1}).

\textbf{Case 4:}
Suppose $p,q,r,s\in X\setminus Y$. Choose sequences $\{p_n\},\{q_n\},\{r_n\},\{s_n\}\subset Y$ such that $p_n\to p$, $q_n\to q$, $r_n\to r$, $s_n\to s$, with all $p_n, q_n, r_n, s_n$ pairwise distinct. Applying the limiting argument, we obtain the inequality (\ref{0.1}).

Thus, in all cases, (\ref{0.1}) holds for all distinct $p,q,r,s\in X$. Therefore $T$ is a GQCRR-type mapping on $X$.
By Theorem \ref{theorem5.1}, we conclude that $p^*$ is a fixed point of $T$.
\end{proof}
\begin{example}
Let \(X=(0,1]\) be equipped with the usual metric $d$. Define \(T:X\to X\) by
\[
T(x)=
\begin{cases}
\dfrac{x+1}{8}, & x\in(0,1),\\[6pt]
\dfrac{1}{8}, & x=1.
\end{cases}
\]
Then $T$ is a CRR-type mapping with $a=\frac{1}{8},\;b=0,\;c=\frac{1}{4}$. By Proposition \ref{prop1}, $T$ is a GQCRR-type mapping with $\alpha=a=\frac{1}{8}$ and $\lambda=b+c=\frac{1}{4}$. The sequence of iterates $p_0\in X$ converges to $\frac{1}{7}$ and $T$ is continuous at $\frac{1}{7}$. From Theorem \ref{theorem5.1}, $\frac{1}{7}$ is a fixed point of $T$.
\end{example}

\begin{example}
Let \(X=(0,1]\) with the usual metric $d$ and let \(Y=(0,1)\), which is dense in \(X\).
Define \(T:X\to X\) by
\[
T(x)=\frac{x+1}{8}\qquad \forall~x\in X.
\]
Then $T$ is a CRR-type mapping with $a=\frac{1}{8},\;b=c=\frac{1}{16}$. By Proposition \ref{prop1}, $T$ is a GQCRR-type mapping with $\alpha=a=\frac{1}{8}$ and $\lambda=b+c=\frac{1}{8}$. The sequence of iterates $p_0\in X$ converges to $\frac{1}{7}$ and $T$ is continuous at $\frac{1}{7}$. From Theorem \ref{theorem5.2}, $\frac{1}{7}$ is a fixed point of $T$.
\end{example}

\section*{Conclusion}
In this work, we have introduced and investigated a new class of mapping, termed as generalized quadrilateral Ćirić-Reich-Rus type mappings. This notion encompasses several well-known concepts, including Ćirić-Reich-Rus-type mappings, generalized Ćirić-Reich-Rus-type mappings, perimetric contraction on quadrilaterals, and Kannan-type perimetric contractions on quadrilaterals. We establish that in complete metric spaces, any GQCRR mapping devoid of periodic points of prime period 2 and 3 necessarily admits a fixed point, with the number of such fixed points being at most three. This non-uniqueness feature distinguishes our class from classical contraction principles while preserving the existence guarantee. A particularly significant finding is that while these mappings may exhibit discontinuities throughout their domain, they must necessarily be continuous at their fixed points, establishing an important regularity property.

Further, we have extended our results in two crucial directions. First, by incorporating asymptotic regularity, we relaxed the parameter constraints while preserving fixed point existence. Second, we established fixed point theorems in incomplete metric spaces. In one scenario, we assume continuity at the limit point and in another, we assume that the mapping is continuous on the entire space and that the contraction condition holds on a dense subset, thereby broadening the applicability of our framework. Several illustrative examples are provided to validate the theoretical results and demonstrate the independence of our new class from existing notions. This work opens up various new research avenues, such as extending these results to more general frameworks like semi-metric spaces via triangle functions, modular function spaces, or multivalued settings would represent natural progressions of this research direction.

\section*{Conflict of Interest}
The authors declare that there is no conflict of interest regarding the publication of this paper.

\end{document}